\documentclass[12pt]{amsart}
\usepackage{amssymb}
\usepackage{amsfonts}
\usepackage{mathtools}
\usepackage{amsthm}
\usepackage{enumitem}
\usepackage[url=false,maxbibnames=99,isbn=false,giveninits=true,doi=false]{biblatex}
\newcommand{\NN}{\mathbb{N}}

\newcommand{\QQ}{\mathbb{Q}}

\newcommand{\ZZ}{\mathbb{Z}}

\newcommand{\cL}{\mathcal{L}}

\DeclareMathOperator{\GL}{GL}

\newcommand{\oag}{\mathrm{oag}}

\newcommand{\K}{\mathrm{K}}
\newcommand{\Def}{\mathrm{Def}}

\theoremstyle{theorem}
\newtheorem{theorem}{Theorem}

\newtheorem{lemma}[theorem]{Lemma}

\theoremstyle{remark}
\newtheorem{remark}{Remark}

\title[The Grothendieck ring of a non-divisible OAG is trivial]{The Grothendieck ring of a non-divisible ordered abelian group is trivial}

\author[B.~Boissoneau]{Blaise Boissonneau}
\address{Blaise Boissonneau. Heinrich Heine University Düsseldorf, Faculty of Mathematics and Natural Sciences, Universitätsstr.~1, 40225 Düsseldorf, Germany.}
\email{blaise.boissonneau@hhu.de}
\urladdr{https://blaxapate.github.io/}

\author[M.~Stout]{Mathias Stout}
\address{Mathias Stout, McMaster University, Department of Mathematics and Statistics, 1280 Main Street West, Hamilton, Ontario L8S 4L8, Canada}
\email{stoutm1@mcmaster.ca}
\urladdr{https://sites.google.com/view/mathias-stout}

\author[F.~Vermeulen]{Floris Vermeulen}
\address{Floris Vermeulen. University of Münster, Mathematics Münster, Einsteinstrasse 62, 48149 Münster, Germany}
\email{florisvermeulen.math@gmail.com}
\urladdr{https://sites.google.com/view/floris-vermeulen}

\thanks{We thank Pierre Touchard for many interesting discussions around Grothendieck rings. M.S.\ is supported by McMaster University and the Fields Institute. F.V.\ is supported by the Humboldt Foundation.
The main result of this paper was obtained during the Decidability, Definability and Computability program which took place at the Hausdorff Insitute in Bonn, funded by the Deutsche Forschungsgemeinschaft (EXC-2047/1 -- 390685813).}

\date{\today}

\begin{document}
	
\begin{abstract}
	We consider the model-theoretic Grothendieck ring of definable sets in ordered abelian groups. 
	It is well-known that $\K \QQ \cong \ZZ[T]/(T^2 + T)$ and  $\K \ZZ =0$, but surprisingly little is known about other cases.
	We present a short computation which shows that they all collapse: $\K G = 0$, unless $G$ is divisible.
\end{abstract}

\maketitle
	
We consider an ordered abelian group $G$, equipped with the natural $\cL_{\oag} = \{0,+,-,<\}$-structure. 
By \emph{definable in $G$} we mean $\cL_\oag(G)$-definable (i.e.\ with parameters). 
Write $\Def(G)$ for the collection of all definable sets in $G^n$, for $n \in \NN$.

We study the model-theoretic \emph{Grothendieck ring} $\K G$ of $G$, as introduced by Krajiček-Scanlon in \cite{KS00}.
The underlying group is given by
\begin{itemize}
	\item Generators: $[A]$ for all $A \in \Def(G)$.
	\item Relations: 
	\begin{enumerate}
		\item $[A] = [B]$ if there exists a definable bijection $A \xrightarrow{\sim} B$,
		\item $[A \setminus B]  +[B] = [A] $ for all $A,B \in \Def(G)$ with $B \subseteq A$.
	\end{enumerate}
\end{itemize}
Then $\K G$ is naturally a ring for the multiplication induced by the cartesian product: $[A] \cdot [B] = [A \times B]$.
It has a zero element $[\emptyset] = 0$, and for any $a \in G$ it holds that $[\{a\}] = 1$.
Note that this construction works more generally for any first-order structure.

It is almost immediate to see that $\K \ZZ = 0$. Indeed, we have that $[\ZZ_{>0}] + [\{0\}] = [\ZZ_{>0}]$, by using the translation $x \mapsto x + 1$.
More generally, this argument works for any discretely ordered abelian group. 
On the other hand, it is known that if $G$ is divisible then $\K G = \ZZ[T]/(T^2 + T)$ by results of Kageyama-Fujita~\cite{KF06}. 

Somewhat surprisingly, not much is known about the intermediate case, despite the interest in the Grothendieck rings of adjacent structures, such as valued fields \cite{CH01,Clu04,HK06,CL08,SV25} (e.g.\ for the purposes of motivic integration).
To the best of our knowledge, the only available result in this setting is due to Bhardwaj and Moonen~\cite{BM26}, proving that $\K G$ is always a quotient of $(\ZZ/q\ZZ)[T]/(T^2 + T)$, where $q$ is the largest odd prime dividing $p-1$ for all primes $p$ for which $G$ is not $p$-divisible.
In fact, more is true:

\begin{theorem}\label{th:KG_triv}
	The Grothendieck ring of definable sets of any non-divisible abelian ordered group is trivial: $\K G = 0$.
\end{theorem}

This theorem follows from the calculations in Lemma~\ref{le:compute_ring} below, which builds upon the work in~\cite{KF06} and \cite{BM26}. 
We present these calculations in full to keep this note self-contained.

If $G$ is any ordered divisible abelian group, we write $G^{\mathrm{div}}$ for its divisible closure. 
For $a,b \in G^{\mathrm{div}} \cup \{\pm \infty\}$, we write $(a,b)$ for the open interval $\{x \in G \mid a < x < b \}$. 
Note that this is an $\cL_\oag(G)$-definable set even when $a,b \in G^{\mathrm{div}} \setminus G$.
We also name two types of special classes inside $\K G$, for $a \in G^{\mathrm{div}}_{>0}$:
\begin{enumerate}
	\item $T_a \coloneqq [(0,a)]$,
	\item $S_a \coloneqq [(a,\infty)]$.
\end{enumerate}
\begin{lemma}\label{le:compute_ring}
	The following identities hold in $\K G$, for any ordered abelian group $G$,
	\begin{enumerate}
		\item \label{it:s2} $S_a^2 = 2 S_0 S_a + S_a$ for all $a \in G^{\mathrm{div}}_{>0}$,
		\item \label{it:t2} $T_a^2 = - T_a$ for all $a \in G^{\mathrm{div}}_{> 0} \setminus G$,
		\item \label{it:Ta} $T_{a} = - 1$ for all $a \in G_{>0}$.
	\end{enumerate}
	Moreover, if there exists a minimal prime $p$ for which $G$ is not $p$-divisible, then
	\begin{enumerate}[resume]
		\item \label{it:Tb} $2 T_{b} = -1$ if $b \in G^{\mathrm{div}}_{>0} \setminus G$, and $pb \in G$.
	\end{enumerate}
\end{lemma}
\begin{proof}
	(\ref{it:s2}) We break up $(a,\infty)^2$ into
	\[ \{ (x,y) \mid  x>y>a \} \sqcup \{ (x,x) \mid x > a \} \sqcup \{ (x,y) \mid y > x > a \},\]
	and then use the definable bijection
	\[ \begin{array}{ccc}
		\{ (x,y) \mid  x>y>a \} &\rightarrow&(0,\infty)\times(a,\infty) ,\\
		(x,y)&\mapsto&(x-y,y).
	\end{array}  \]

	(\ref{it:t2}) We use (\ref{it:s2}) to compute
	\begin{align*}
		T_a^2 	&= (S_0 - S_a)^2 \\
		&= S_0^2 - 2 S_0 S_a  + S_a^2 \\
		&= -S_0 - 2 S_0 S_a + 2 S_0 S_0 + S_a \\
		&= S_a - S_0 \\
		&= - T_a. 
	\end{align*}
	
	(\ref{it:Ta}) We observe that
	\[[(0,\infty)] =  [(0,a)] + [\{a\}] + [(a,\infty)],\]
	where $(a,\infty)$ is definably isomorphic to $(0,\infty)$ since $a \in G$.
	
	(\ref{it:Tb}) Note that
	\begin{align*}
		-1 	&= [(0,pb)]  \\
		&= [(0,(p-1)b)]  + [((p-1)b, p b)] \\
		&= [(0,b)] + [(-b,0)] \\
		&= [(0,b)] + [(0,b)],
	\end{align*}
	where we used that $x \mapsto x/(p-1)$ and $x \mapsto x - pb$ are $\cL_\oag(G)$-definable bijections.
\end{proof}
\begin{proof}[Proof of Theorem~\ref{th:KG_triv}]
	Let $G$ be non-divisible. By Lemma~\ref{le:compute_ring}(\ref{it:Tb}) there exists some $b \in G^{\mathrm{div}}\setminus G$ such that $2T_b = -1$.
	We therefore have that $4T_b^2 = 1$.
	On the other hand, relation (\ref{it:t2}) shows that $4T_b^2 = -4T_b = 2$.
	Hence $2=1$ which shows that $1 = 0$.
\end{proof}

Triviality of the Grothendieck ring precisely means that the onto-pigeonhole principle fails, see~\cite[Cor.\,3.2]{KS00}. In other words, there exists a definable set $X$ such that $X$ is in definable bijection with $X \sqcup \{0\}$.
The computations in Lemma~\ref{le:compute_ring} can be traced back to show that this holds for
\[  X = \bigsqcup_{i =1}^6 (0,\infty)^2 \sqcup \bigsqcup_{i=1}^8 (0,\infty) \times (b,\infty), \]
where $b \in G^{\mathrm{div}}_{>0} \setminus G$, but $p b \in G$ where $p$ is the minimal prime such that $G$ is not $p$-divisible.

\begin{remark}
	Although the Grothendieck rings of all ordered abelian groups are now known, many interesting challenges remain.
	One could rather attempt to compute the Grothendieck semiring, which amounts to classifying all definable sets up to definable bijections.
	This seems out of reach for general ordered abelian groups as it is already nontrivial in the discrete~\cite{CH18} and divisible case~\cite{vdD98}.
	A more reasonable challenge would be to consider $\mathrm{dp}$-minimal groups of rank 1.
	
	In another direction, there are many variants of $\K G$ that come up naturally, e.g.\ in the context of motivic integration (\cite{HK06,SV25}) whose full structure is not yet known.
	Essentially, these variants are given by restricting the notion of morphism. We mention three cases of interest:
	\begin{enumerate}
		\item \label{it:KG[*]} The graded Grothendieck semiring $\K G[*]$, where one only considers definable maps between subsets of the same ambient dimension.
		\item \label{it:GLn} The ring obtained by further restricting to piecewise $\GL_n(\ZZ) \ltimes G^n$-maps. 
			More generally, one may replace $\ZZ$ by any subring of $\QQ$ generated by inverting some set of primes over $\ZZ$.
		\item Measured versions of the above rings: the basic objects are now pairs $(X,\mu)$ consisting of a definable set and a definable map $\mu \colon X \to G$. A morphism $f \colon (X,\mu) \to (Y,\omega)$ consists of a definable bijection $f \colon X \to Y$ which is piecewise $\GL_n(\ZZ) \ltimes G^n$ and which is measure-preserving, meaning that
			\[x_1 + \dots + x_n + \mu(x) = f_1(x) + \dots + f_n(x) + \omega(x) .\] 
	\end{enumerate}
\end{remark}

\printbibliography	
\end{document}